\documentclass[11pt]{amsart}
\usepackage{amsfonts}

\newtheorem{thm}{Theorem}[section]
\newtheorem{lem}[thm]{Lemma}
\newtheorem{cor}[thm]{Corollary}
\newtheorem{prop}[thm]{Proposition}

\newtheorem{defprop}[thm]{Definition-Proposition}

\theoremstyle{definition}
\newtheorem{exmp}[thm]{Example}
\newtheorem{defn}[thm]{Definition}

\theoremstyle{remark}
\newtheorem{rem}[thm]{Remark}

\newcommand{\Z}{{\mathbb{Z}}}

\newcommand{\Ext}{\operatorname{Ext}}
\newcommand{\Hom}{\operatorname{Hom}}
\newcommand{\Endhom}{\operatorname{End}}

\newcommand{\add}{\mathrm{add}\; }

\newcommand{\ind}{\mathrm{ind}\text{-} }
\newcommand{\modu}{\mathrm{mod}\text{-} }

\newcommand{\tiltq}{\overrightarrow{\mathcal{T}}}

\begin{document}

\title[A pre-projective part of  tilting quivers]{A pre-projective part of  tilting quivers of certain path algebras}
\author{Ryoichi Kase}
\address{Department of Pure and Applied Mathematics
Graduate School of Information Science and Technology
,Osaka University, Toyonaka, Osaka 560-0043, Japan}
\email{r-kase@cr.math.sci.osaka-u.ac.jp}
\date{}

\maketitle
\begin{abstract}
D.Happel and L.Unger defined a partial order on the set of basic tilting modules.
 We study the poset of basic pre-projective tilting modules over path algebra of infinite type.
  First we will give a criterion for Ext-vanishing  for pre-projective modules. And by using this we will give a combinatorial characterization of the poset of basic pre-projective tilting modules. Finally we will see a structure of a pre-projective part of  tilting quivers.
\end{abstract}

\section*{Introduction}
Tilting theory first appeared in the article by Brenner and Butler   \cite{BB}. In this article the notion of a tilting module for finite dimensional algebra was introduced. Tilting theory  now appear  in many areas of mathematics, for example algebraic geometry, theory of algebraic groups and algebraic topology. Let $T$ be a tilting module for finite dimensional algebra $A$ and let $B=\Endhom_{A}(T)$. Then Happel showed that the two bounded derived categories
$\mathcal{D}^{\mathrm{b}}(A)$ and $\mathcal{D}^{\mathrm{b}}(B)$
are equivalent as triangulated category. This is one of the most important result in representation theory of finite dimensional algebras. And so classifying many tilting modules is an important problem.
 
  Theory of tilting-mutation introduced by Riedtmann and Schofield is one of the approach to this problem. Riedtmann and Schofield defined the tilting quiver related with tilting-mutation.  Happel and Unger defined the partial order on the set of basic tilting modules  and showed that tilting quiver is coincided with Hasse-quiver of this poset. And now these combinatorial structure are studied by many authors.

In this paper we use the following notations. Let $A$ be a
 finite dimensional algebra over an algebraically closed
 field $k$, and let $\modu A$ be the category of finite 
 dimensional right $A$-modules. For $M\in \modu A$ we denote by pd$_{A}M$ the projective
 dimension of $M$, and by $\add M$ the full subcategory
 of direct sums of direct summands of $M$. Let $Q=(Q_{0},Q_{1})$ be a finite connected quiver without loops and cycles, 
 and $Q_{0}$ (resp.\;$Q_{1}$) be the set of vertices (resp.\;arrowss) of $Q$. For any $\alpha:x\rightarrow y $ set $s(\alpha):=x $ and $t(\alpha):=y$. 
  We denote
 by $kQ$ the path algebra of $Q$ over $k$.
   For any paths $w:a_{0}\stackrel{\alpha_{1}}{\rightarrow}a_{1}\stackrel{\alpha_{2}}{\rightarrow}\cdots \stackrel{\alpha_{r}}{\rightarrow} a_{r}$ and $w^{'}:b_{0}\stackrel{\beta_{1}}{\rightarrow}b_{1}\stackrel{\beta_{2}}{\rightarrow}\cdots \stackrel{\beta_{s}}{\rightarrow} b_{s}$, \[w\cdot w^{'}:=\left\{\begin{array}{ll}
 a_{0}\stackrel{\alpha_{1}}{\rightarrow}a_{1}\stackrel{\alpha_{2}}{\rightarrow}\cdots \stackrel{\alpha_{r}}{\rightarrow} a_{r}=b_{0}\stackrel{\beta_{1}}{\rightarrow}b_{1}\stackrel{\beta_{2}}{\rightarrow}\cdots \stackrel{\beta_{s}}{\rightarrow} b_{s} & \mathrm{if\ }a_{r}=b_{0} \\ 
 0 & \mathrm{if\ }a_{r}\neq b_{0},
 \end{array}\right.\]
 in $kQ$. For any module $M\in \modu kQ$ we denote by $|M|$ the number of
 pairwise non isomorphic indecomposable direct summands of $M$.

We denote by $\overrightarrow{\mathcal{T}}(Q)$ the tilting quiver over $Q$\;(see Definition\;\ref{dtiltq} ).
In this paper we will see a structure of $\overrightarrow{\mathcal{T}}_{\mathrm{p}}(Q)$ the pre-projective part of $\overrightarrow{\mathcal{T}}(Q)$\;(i.e.
the full sub-quiver of $\overrightarrow{\mathcal{T}}(Q)$ having pre-projective basic tilting modules as the set of vertices) 
when $Q$ satisfies some conditions (see Section\;2).  
 Now a module $M\in \modu A$  is pre-projective if there exists some projective module 
 $P$ and non-negative integer $r$ s.t.\;$M\simeq \tau^{-r}P$, where $\tau$ is 
the Auslander-Reiten translation.

We give an outline of this paper. 
In Section\;1, following\cite{HU1},\cite{HU2},\cite{HU3},\cite{HU4}, we recall some definitions and properties used in this paper.
In section\;2 we give our main Theorem. In Section\;3 
we give a criterion for Ext-vanishing for pre-projective modules. 
More precisely we introduce a function $l_{Q}$ from $Q_{0}\times Q_{0}$ to $\Z_{\geq 0}$ 
such that $\Ext_{kQ}^{1}(\tau^{-r_{i}}P(i),\tau^{-r_{j}}P(j))=0$ 
if and only if $r_{i}\leq r_{j}+l_{Q}(j,i)$ and by using this fact we prove our main Theorem.
In Section\;4 we will see a structure of $\overrightarrow{\mathcal{T}}_{\mathrm{p}}(Q)$ in the case $l(Q):=\mathrm{max}\{l_{Q}(x,y)\mid x,y\in Q_{0}\}\leq 1$.  

In this paper we identify two quivers $Q$ and $Q^{'}$ if $Q$ is isomorphic to $Q^{'}$ as a quiver.  
   
\section{Preliminaries}

In this section, following \cite{HU1},\cite{HU2},\cite{HU3} and \cite{HU4}, we will recall the definition of tilting module and basic results for combinatorics of the set of tilting modules.

\begin{defn}
A module  $T\in \modu kQ$ is tilting module if,\\
$(1)$\;$\Ext^{1}_{kQ}(T,T)=0$,\\
$(2)$\;$|T|=\#Q_{0}.$ 
 
\end{defn}

\begin{rem}
Generally we call a module $T$ over a finite dimensional algebra $A$ a tilting module if (1)\;its projective dimension
is at most $1$, (2)\;$\Ext_{A}^{1}(T,T)=0$ and (3)\;there is a exact sequence,
\[0\rightarrow A_{A}\rightarrow T_{0}\rightarrow T_{1}\rightarrow 0,\]
with $T_{i}\in \add T$. In the case $A$ is hereditary it is well-known that this definition is equivalent to our definition. 
 
\end{rem}

We denote by $\mathcal{T}(Q)$ the set of (isomorphism classes of) basic tilting modules in $\modu kQ$.

\begin{defprop}\cite[Lemma\;2.1]{HU2} Let $T,T^{'}\in \mathcal{T}(Q)$. Then the following relation $\leq$ define  a partial order on $\mathcal{T}(Q)$,
\[T\geq T^{'}\stackrel{\mathrm{def}}{\Leftrightarrow}\Ext^{1}_{kQ}(T,T^{'})=0.\]
\end{defprop}

\begin{defn}
\label{dtiltq}
The tilting quiver $\tiltq (Q)$ is defined as follows,\\
(1)\;the set of vertices $\tiltq (Q)_{0}:=\mathcal{T}(Q)$,\\
(2)\;$T\rightarrow T^{'}$ in $\tiltq (Q)$ if  $T\simeq M\oplus X$
, $T^{'}\simeq M\oplus Y$ for some $X,Y\in \ind kQ$, $M\in \modu kQ$ and there is a non split exact sequence,
\[0\rightarrow X \rightarrow M^{'}\rightarrow Y\rightarrow 0, \]
with $M^{'}\in \add M$.

\end{defn}

The following results give interesting properties of tilting quivers.

\begin{thm}\cite[Theorem\;2.1]{HU1}
The tilting quiver $\tiltq (Q)$ is coincided with the $\mathrm{Hasse}$-quiver
of $(\mathcal{T}(Q),\leq)$.
\end{thm}

\begin{rem}
In this paper we define the Hasse-quiver $\overrightarrow{P}$ of  (finite or infinite)  poset $(P,\leq)$ as follows,\\
(1)\;the set of vertices $\overrightarrow{P}_{0}:=P$,\\
(2)\;$x\rightarrow y$ in $\overrightarrow{P}$ if $x>y$ and there is no $z\in P$ such that $x>z>y$.

\end{rem}

\begin{prop}\cite[Corollary\;2.2]{HU1} If $\tiltq (Q)$ has a finite component $\mathcal{C}$,
 then $\tiltq (Q)=\mathcal{C}$.
\end{prop}

\begin{thm}
 \label{HU-thm}
\cite[Theorem\;6.4]{HU3}
 If $Q$ is connected and has no multiple arrows, then $Q$ is uniquely determined by $(\mathcal{T}(Q),\leq )$. 
\end{thm}

Let $M$ be a basic partial tilting module and $\mathrm{lk}(M):=\{T\in\mathcal{T}\mid M\in \add T\}$. Then we denote by 
$\overrightarrow{\mathrm{lk}}(M)$ the full sub-quiver of $\overrightarrow{\mathcal{T}}(Q)$
having $\mathrm{lk}(M)$ as the set of vertices (see \cite{HU4}).
\begin{prop}
\cite[Theorem\;4.1]{HU4}
If $M$ is faithful, then $\overrightarrow{\mathrm{lk}}(M)$ is connected.
\end{prop}

\section{Main result}
In this section we give our main result for a structure of $\tiltq_{\mathrm{p}}(Q)$.
Let $Q$ be a quiver with $n$ vertices.  For any $x\in Q_{0}$
we denote by $P(x)$ an indecomposable projective module associated with $x$. 
We denote by $\mathcal{T}_{\mathrm{p}}(Q)$ the set of basic pre-projective tilting modules, 
and for a basic pre-projective partial tilting module $M$ define $\mathrm{lk}_{\mathrm{p}}(M)$, $\overrightarrow{\mathrm{lk}}_{\mathrm{p}}(M)$ similarly.
\begin{thm}
\label{t}
Assume that $Q$ satisfies the following conditions $(\mathrm{a})$ and $(\mathrm{b})$\\
$(\mathrm{a})\;Q$ has a unique source $s\in Q_{0}$.\\
$(\mathrm{b})\;$For any $x\in Q_{0}$, $\#s(x)+\#t(x)>1$.

Then the following assertions hold,\\
$(1)$\;$\mathcal{T}_{\mathrm{p}}(Q)$ is a disjoint union of $\mathrm{lk}_{\mathrm{p}}(\tau^{-r}P(s))$ for all $r\geq 0$.\\
$(2)$\;$\tau^{-r}$ gives a quiver isomorphism \[\tau^{-r}:\overrightarrow{\mathrm{lk}}_{\mathrm{p}}(P(s))\simeq \overrightarrow{\mathrm{lk}}_{\mathrm{p}}(\tau^{-r}P(s)).\]\\
$(3)\;\#\mathrm{lk}_{\mathrm{p}}(P(s))\leq 2^{n-1}$.\\
$(4)\;\overrightarrow{\mathrm{lk}}_{\mathrm{p}}(P(s))$ is a connected quiver,\\
$(5)$ Let $T\in \mathrm{lk}_{\mathrm{p}}(\tau^{-r}P(s))$ and $T^{'}\in \mathrm{lk}_{\mathrm{p}}(\tau^{-r^{'}}P(s))$. If there is an arrow $T\rightarrow T^{'}$, then $r^{'}-r$ is $0$ or $1$.\\
$(6)\;\overrightarrow{\mathcal{T}}_{\mathrm{p}}(Q)$ is a connected quiver. 
\end{thm}
\section{A proof of Theorem\;\ref{t}}
In this section we prove Theorem\;\ref{t}.
Let $Q$ be a quiver with $Q_{0}=\{0,1,\cdots,n-1\}$.
Without loss of generality we assume that if $\exists \alpha:i\rightarrow j$ in $Q$ 
then $i>j$.
For $x\in \{0,1,\cdots,n-1\}$ and $r>0$, put 
$P(x+rn):=\tau^{-r}P(x)$.  And for any $x\in Q_{0}$
put $s(x):=\{\alpha\in Q_{1}\mid s(\alpha)=x\}$ and $t(x):=\{\beta\in Q_{1}\mid t(\alpha)=x\}$. 

We collect basic properties of the Auslander-Reiten translation.
\begin{prop}
\label{iso}
$($\cite{ASS}$)$
Let $A=kQ$ be a path algebra, and $M,N\in \modu A$ be a non-injective
right $A$-modules. Then,
\[\Hom_{A}(M,N)\simeq \Hom_{A}(\tau^{-1}M,\tau^{-1}N). \]  
\end{prop}

\begin{prop}
$($\cite{ARS}$)$\;(Auslander-Reiten duality)\;
Let $A=kQ$ be a path algebra, and $M,N\in \modu A$.  Then,
\[D\Hom_{A}(M,N)\simeq \Ext_{A}^{1}(N,\tau M). \]
\end{prop} 

\begin{prop}
\label{dim}
$($\cite{G}$)$
Let $A=kQ$ be a path algebra and $M\in \ind A$. Then for any indecomposable non-projective module $X$ and almost
split sequence
\[0\rightarrow \tau X\rightarrow E\rightarrow X\rightarrow 0,\]
we get
\[\mathrm{dim}\Hom(M,\tau X)-\mathrm{dim}\Hom(M,E)+\mathrm{dim}\Hom(M,X)=\left\{\begin{array}{ll}
1 & X\simeq M \\ 
0 & \mathrm{otherwise.}
\end{array}
\right.\]
\end{prop}
 Let $d_{a}(b):=\mathrm{dim}$\;$\Ext^{1}_{kQ}(P(b),P(a))$. If $Q$ satisfies
 the condition (b) of Theorem\;\ref{t}, then $kQ$ is representation infinite and pre-projective part of its AR-quiver is $\Z_{\leq 0}Q$ (cf.\cite{ASS}).
 So from above proposition and AR-duality, we get the following,
\[d_{a}(x+rn)=\left\{\begin{array}{ll}
0 & x+rn<a+n \\ 
1 & x+rn=a+n \\ 
\sum_{\alpha\in s(x)}d_{a}(t(\alpha)+rn)+\sum_{\beta\in t(x)}d_{a}(s(\beta)+(r-1)n) & \\
\ \ \ \ \ \ \ \ \ \ \ \ \ \ \ \ \ \ \ \ \ \ \ \ \ \ \ \ \ \;-d_{a}(x+(r-1)n) & x+rn>a+n, \\
\end{array} \right.\]
\begin{lem}
\label{l}
Assume $Q$ satisfies the condition $(b)$ of Theorem\;\ref{t}. Let $\gamma : x\rightarrow y$ in $Q$. Then, 
\[d_{a}(y+rn)\leq d_{a}(x+rn)\leq d_{a}(y+(r+1)n)\]
for any $r\geq 0$.
\end{lem}
\begin{proof}
We can assume $a<n$, and use induction on $y+rn$.

$(y+rn< n\;(\mathrm{i.e.\;r=0}))$: In this case $d_{a}(y)=d_{a}(x)=0$.

$(n\leq y+rn<a+n)$: In this case $d_{a}(y+rn)=0$ and
\[d_{a}(y+(r+1)n)=\sum_{\alpha\in s(y)}d_{a}(t(\alpha)+(r+1)n)+\sum_{\beta\in t(y)}d_{a}(s(\beta)+rn)\geq d_{a}(x+rn).\]

$(y+rn\geq a+n)$: In this case
\[\begin{array}{lll}
d_{a}(x+rn) & = & \sum_{\alpha\in s(x)}d_{a}(t(\alpha))+\sum_{\beta\in t(x)}d_{a}(s(\beta))-d_{a}(x+(r-1)n) \\ 
 & = & \sum_{\alpha\in s(x)\setminus \{\gamma\}}d_{a}(t(\alpha)+rn)+\sum_{\beta\in t(x)}d_{a}(s(\beta)+(r-1)n) \\
 & & -d_{a}(x+(r-1)n)+d_{a}(y+rn) \\ 
(\ast)\cdots & \geq & d_{a}(y+rn) \\
\end{array}
\]
(remark. $(\ast)$ is followed by (b) and hypothesis of induction.)
and similarly we can get $d_{a}(y+(r+1)n)\geq d_{a}(x+rn)$.
\end{proof}
Let $\tilde{Q}$ be a quiver obtained from $Q$ by adding new edge
$-\alpha:y\rightarrow x$ for any $\alpha:x\rightarrow y$.  
For a path $w:x_{0}\stackrel{\alpha_{1}}{\rightarrow}x_{1}\stackrel{\alpha_{2}}
{\rightarrow}\cdots \stackrel{\alpha_{r}}{\rightarrow} x_{r}$ in $\tilde{Q}$, put
$c^{+}(w):=\#\{t\mid \alpha_{t}\in Q_{1}\subset \tilde{Q}_{1}\}$, and let $l_{Q}(i,j):=\mathrm{min}\{c^{+}(w)\mid w:\text{path\ from}\  i\  \mathrm{to}\  j\  \mathrm{in}\  \tilde{Q}\}$ (set $l_{Q}(i,i)=0$ for
any $i$ ).
\begin{prop}
\label{p}
If $Q$ satisfies the condition $(b)$ of Theorem\;\ref{t},
 then
 \[\Ext_{kQ}^{1}(\tau^{-r}P(i),\tau^{-s}P(j))=0\Leftrightarrow r\leq s+l_{Q}(j,i)\]
 \end{prop}
 \begin{proof} 
 $(\Rightarrow)$: Let $w:j=x_{0}\stackrel{\alpha_{1}}{\rightarrow}x_{1}\rightarrow
 \cdots \stackrel{\alpha_{t}}{\rightarrow}x_{t}=i$ be a path s.t.
 $l(j,i):=l_{Q}(j,i)=c^{+}(w)$ and $\{k_{1}<k_{2}<\cdots<k_{l(j,i)}\}=
 \{k\mid \alpha_{k}\in Q_{1}\}.$ If $\exists r>l(j,i)$ s.t.
 $\Ext^{1}_{kQ}(\tau^{-r}P(i),P(j))=0$, then, by Lemma\;\ref{l}, we get
 \[\begin{array}{l}
 0=d_{j}(x_{t}+rn)\geq d_{j}(x_{k_{l(j,i)}}+rn)\geq d_{j}(x_{k_{l(j,i)}-1}+(r-1)n)\\
 \geq \cdots \geq d_{j}(x_{k_{1}}+(r-l(j,i)+1)n)\geq d_{j}(x_{k_{1}-1}+(r-l(j,i))n)\\ 
 \geq d_{j}(j+(r-l(j,i))n))\geq d_{j}(j+(r-l(j,i)-1)n))
 \geq \cdots \geq d_{j}(j+n)>0,
 \end{array}\]
 and this is a contradiction. So if $\Ext_{kQ}^{1}(\tau^{-r}P(i),\tau^{-s}P(j))=0$ with $r > s+l(j,i)$, then by proposition\;\ref{iso}, we get a contradiction.
 
$(\Leftarrow)$ Let $\mathcal{A}(j):=\{(i,r)\mid r\leq l(j,i),\ \Ext^{1}_{kQ}(\tau^{-r}P(i),P(j))\neq 0\}$. If $\mathcal{A}(j)\neq \emptyset$,
 then we can take $r:=\mathrm{min}\{r\mid (i,r)\in \mathcal{A}(j)\ \mathrm{for\ some\ }i\}$ and $i\in Q_{0}$ s.t.\;$(i,r)\in \mathcal{A}(j)$
$(i^{'},r)\notin \mathcal{A}(j)$ for any $i^{'}\leftarrow i$ in $Q$.
Now 
\[0< d_{j}(i+rn)\leq \sum_{\alpha\in s(i)}d_{j}(t(\alpha)+rn)+\sum_{\beta\in t(i)}d_{j}(s(\beta)+(r-1)n)\]
implies that $d_{j}(t(\alpha)+rn)\neq 0 $ for some $\alpha \in s(i)$ or $d_{j}(s(\beta)+(r-1)n)\neq 0$ for some $\beta \in t(i)$.
Note that $r\leq l(j,i)\leq l(j,t(\alpha)),\ l(j,s(\beta))+1$ for any $\alpha \in s(i)$ and $\beta \in t(i)$.
So $d_{j}(t(\alpha)+rn)=0=d_{j}(s(\beta)+(r-1)n)$ for any $\alpha \in s(i)$ and $\beta \in t(i)$ and this is a contradiction. So we get 
$\mathcal{A}(j)=\emptyset.$ 

We assume that $\exists i\in Q_{0}$ and $\exists r,s\in \Z_{\geq 0}$
 s.t.\;$r\leq s+l(j,i)$ and \[\Ext_{kQ}^{1}(\tau^{-r}P(i),\tau^{-s}P(j))\neq 0.\] 
 If $r< s$, then proposition\;\ref{iso} shows $\Ext_{kQ}^{1}(\tau^{-r}P(i),\tau^{-s}P(j))= 0$. So $r\geq s$. Now
 proposition\;\ref{iso} implies $(i,r-s)\in \mathcal{A}(j)$ and this is
 a contradiction.    
   
 \end{proof}
\begin{lem}
 \label{l2}
 Let $T\in \mathcal{T}_{\mathrm{p}}(Q)$ and $T\leq T^{'}\in \mathcal{T}(Q)$.
Then $T^{'}\in \mathcal{T}_{\mathrm{p}}(Q)$. In particular $\overrightarrow{\mathcal{T}}_{\mathrm{p}}(Q)$ $(\mathrm{resp}$.\;$\overrightarrow{\mathrm{lk}}_{\mathrm{p}}(M))$ is a 
Hasse-quiver of $(\mathcal{T}_{\mathrm{p}},\leq)$ $(\mathrm{resp}$.\;$(\mathrm{lk}_{\mathrm{p}}(M),\leq ))$. 
\end{lem}
 \begin{proof}
Let $X$ be an indecomposable direct summand of $T^{'}$. If $X$ is not pre-projective, then $\Ext^{1}_{kQ}(\tau^{-r}P,X)\simeq \Ext^{1}_{kQ}(P,\tau^{r}X)=0 $ 
for any projective module $P$. So $\Ext^{1}_{kQ}(T,X)=0$.
Since $\Ext^{1}_{kQ}(X,T)=0$, we get $X\in \add T$. This is a contradiction.

\end{proof} 
\begin{lem}
Let $T=\oplus\tau^{-r_{i}}P(i),T^{'}=\oplus\tau^{-r^{'}_{i}}P(i)\in \mathcal{T}_{\mathrm{p}}(Q)$. 
If $T\rightarrow T^{'}$ in $\overrightarrow{\mathcal{T}}_{\mathrm{p}}(Q)$\;$($so in $\overrightarrow{\mathcal{T}}(Q))$ then
$\exists i\;\mathrm{s.t.}\;r^{'}_{i}=r_{i}+1$ and $r^{'}_{j}=r_{j}\;(\forall j\neq i)$.  

\end{lem}

\begin{proof}
From tilting theory there exists some $i\;$s.t.$\; r_{i}<r^{'}_{i}$ and $r_{j}=r^{'}_{j}\;(\forall j\neq i).$ Assume that $r^{'}_{i}=r_{i}+t.$
 Then proposition\;\ref{p} shows 
 \[r_{j}-l(i,j)\leq r_{i}<r_{i}+t\leq r_{j}+l(j,i)\;(\forall j\neq i),\]
 and this implies $T^{''}:=\tau^{-r_{i}-1}P(i)\oplus(\oplus_{j\neq i} \tau^{-r_{j}}P(j))\in \mathcal{T}_{\mathrm{p}}(Q)$. 
 Since $T>T^{''}\geq T^{'}$, we get $T^{'}=T^{''}$.
   
\end{proof}
   
For any quiver $Q$ satisfying the condition (b) of Theorem\;\ref{t}, 
put $L(Q):=\{(r_{i})_{i\in Q_{0}}\in \Z_{\geq 0}^{Q_{0}} \mid r_{j}\leq    r_{i}+ l_{Q}(i,j) \}\subset \Z^{Q_{0}}$.
 
Then as an immediate Corollary of proposition\;\ref{p}, we get the following.
 \begin{cor}
 \label{c}
Assume $Q$ satisfies the conditions $(b)$ of Theorem\;\ref{t}. Then 
\[(r_{i})_{i\in Q_{0}}\mapsto \oplus_{i=0}^{n-1}\tau^{-r_{i}}P(i)\]
induces an isomorphism of posets,
\[(L(Q),\leq^{op})\simeq (\mathcal{T}_{\mathrm{p}}(Q),\leq)\]
where  $(r_{i})_{i\in Q_{0}}\geq^{op} (r^{'}_{i})_{i\in Q_{0}}
\stackrel {\mathrm{def}} {\Leftrightarrow} r_{i}\leq r_{i}^{'}$ for any $i\in Q_{0}$.  

In particular $\overrightarrow{\mathcal{T}}_{\mathrm{p}}(Q)$ is a full sub-quiver of 
Hasse-quiver of $(\Z^{Q_{0}},\leq^{op})$.
\end{cor}

\begin{prop}
\label{p2}
Let $i\in Q_{0}$ and $T(i):=\oplus_{j\in Q_{0}}\tau^{-l(i,j)}P(j)$.
Then $T(i)$ is a unique minimal element of $\mathrm{lk}_{\mathrm{p}}(P(i))$.
\end{prop}
\begin{proof}
Let $j,j^{'}\in Q_{0}$. By definition of $l_{Q}$, we get 
\[l(i,j)\leq l(i,j^{'})+l(j^{'},j),\]
and this implies $T(i)\in \mathrm{lk}_{\mathrm{p}}(P(i))$.

Now let $T:=\oplus_{j\in Q_{0}} \tau^{-r_{j}}P(j)\in \mathrm{lk}_{\mathrm{p}}(P(i))$. Then $r_{j}\leq r_{i}+l(i,j)=l(i,j)$. 
So Corollary\;\ref{c} shows $T\geq T(i)$. 
\end{proof}

Now we can show the Theorem\;\ref{t}.
\begin{proof}
(1)\;This is followed by proposition\;\ref{p}.\\
  (2)\;It is obvious that $\tau^{-r}$ induces an injection 
 \[\overrightarrow{\mathrm{lk}}_{\mathrm{p}}(P(n-1))\rightarrow \overrightarrow{\mathrm{lk}}_{\mathrm{p}}(\tau^{-r}P(n-1))\]
 as a quiver. So it is sufficient to show 
 \[\tau^{-r}:\mathrm{lk}_{\mathrm{p}}(P(n-1))\rightarrow 
 \mathrm{lk}_{\mathrm{p}}(\tau^{-r}P(n-1))\]
 is surjective.
 
 Let $T \in \mathrm{lk}_{\mathrm{p}}(\tau^{-r}P(n-1))$ and by Corollary\;\ref{c} we can
 put $T=\oplus_{i=0}^{n-2}\tau^{-r_{i}}P(i)\oplus \tau^{-r}P(n-1)$.
 Since $l(i,n-1)=0$ $(\forall i)$, proposition\;\ref{p} shows 
 $r_{i}\geq r$ $(\forall i)$. So $\tau^{r} T\in \mathrm{lk}_{\mathrm{p}}(P(n-1))$.\\
 (3)\;For any $i\in Q_{0}$ put $t(i):=\mathrm{max}\{j\in Q_{0}\mid j\rightarrow i\}$. Let $T=\oplus_{i=0}^{n-1}\tau^{-r_{i}}P(i).$ 
Then $T\in \mathrm{lk}_{\mathrm{p}}(P(n-1))$ only if 
 \[ r_{n-1}=0,r_{n-2}\in\{0,1\},\cdots , r_{i}\in\{r_{t(i)},r_{t(i)}+1\},\cdots, r_{0}\in \{r_{t(0)},r_{t(0)}+1\}.\]
 This implies $\#Q_{0}\leq 2^{n-1}$.\\
   (4)\;Let $\mathcal{C}$ be a connected component of $\overrightarrow{\mathrm{lk}}_{\mathrm{p}}(P(n-1))$ containing $A=kQ$.
   If $T\in \mathcal{T}_{\mathrm{p}}\setminus \mathcal{C}$, then there is an
  infinite sequence
\[A=T_{0}\rightarrow T_{1}\rightarrow T_{2}\rightarrow \cdots\]
 in $\mathcal{C}$ (cf.\cite{HU2}), and this is a contradiction.\\  
(5)\;Let $T=\oplus_{i\in Q_{0}}\tau^{-r_{i}}P(i)\in \mathrm{lk}_{\mathrm{p}}(\tau^{-r}P(n-1))$ and $T^{'}=
\oplus_{i\in Q_{0}}\tau^{-r^{'}_{i}}P(i)\in \mathrm{lk}_{\mathrm{p}}(\tau^{-r+t}P(n-1))$.
If $T\rightarrow T^{'}$ in $\overrightarrow{\mathcal{T}}_{\mathrm{p}}(Q)$\;(i.e.\;in $\overrightarrow{\mathcal{T}}(Q)$) 
and $t\neq 0$, then $r_{i}=r^{'}_{i}$\;$(\forall i\in Q_{0}\setminus \{n-1\}$), 
$r_{n-1}\leq r^{'}_{n-1}=r_{n-1}+t$ and $r_{n-1}+t=r^{'}_{n-1}\leq r^{'}_{n-2}=r_{n-2}\leq r_{n-1}+l(n-1,n-2)=r_{n-1}+1$. 
This implies $t=1$.\\
(6)\;Put $T:=(\oplus_{i\leq n-2}\tau^{-r}P(i))\oplus \tau^{-r+1}P(n-1)$ and $T^{'}:=\oplus_{i\leq n-1}\tau^{-r}P(i)$. Then Corollary\;\ref{c} shows 
$T\in \mathrm{lk}_{\mathrm{p}}(\tau^{-r+1}P(n-1))$ and $T^{'}\in \mathrm{lk}_{\mathrm{p}}(\tau^{-r}P(n-1))$ with $T\rightarrow T^{'}$. Now $(4)$ of this Theorem implies 
$\overrightarrow{\mathcal{T}}_{\mathrm{p}}(Q)$ is connected.      
\end{proof}
\begin{exmp}
We give the two examples. Recall that for any two elements $(r_{i})$,   $(r^{'}_{i})$ of $\Z^{n}$ $(n\geq 1)$, $(r_{i})\geq^{op} (r^{'}_{i})$
means $r_{i}\leq r_{i}^{'}$ for any $i$.
   
(1):Consider the following quiver
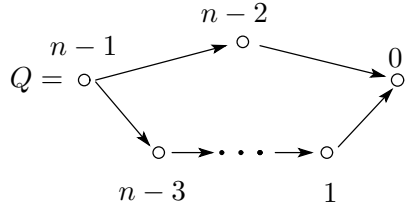
\[
\unitlength 0.1in
\begin{picture}( 21.0200,  9.3000)(  1.7000,-12.7000)
\put(3.5000,-8.0000){\makebox(0,0){$Q=$}}%
%
\special{pn 8}%
\special{ar 600 800 30 30  0.0000000 6.2831853}%
%
\special{pn 8}%
\special{pa 660 810}%
\special{pa 940 1140}%
\special{fp}%
\special{sh 1}%
\special{pa 940 1140}%
\special{pa 912 1076}%
\special{pa 906 1100}%
\special{pa 882 1102}%
\special{pa 940 1140}%
\special{fp}%
%
\special{pn 8}%
\special{ar 1430 600 30 30  0.0000000 6.2831853}%
%
\special{pn 8}%
\special{ar 990 1180 30 30  0.0000000 6.2831853}%
%
\special{pn 8}%
\special{pa 1060 1180}%
\special{pa 1270 1180}%
\special{fp}%
\special{sh 1}%
\special{pa 1270 1180}%
\special{pa 1204 1160}%
\special{pa 1218 1180}%
\special{pa 1204 1200}%
\special{pa 1270 1180}%
\special{fp}%
%
\special{pn 8}%
\special{sh 1}%
\special{ar 1320 1180 10 10 0  6.28318530717959E+0000}%
\special{sh 1}%
\special{ar 1420 1180 10 10 0  6.28318530717959E+0000}%
\special{sh 1}%
\special{ar 1520 1180 10 10 0  6.28318530717959E+0000}%
\special{sh 1}%
\special{ar 1520 1180 10 10 0  6.28318530717959E+0000}%
\special{sh 1}%
\special{ar 1520 1180 10 10 0  6.28318530717959E+0000}%
%
\special{pn 8}%
\special{pa 1600 1180}%
\special{pa 1790 1180}%
\special{fp}%
\special{sh 1}%
\special{pa 1790 1180}%
\special{pa 1724 1160}%
\special{pa 1738 1180}%
\special{pa 1724 1200}%
\special{pa 1790 1180}%
\special{fp}%
%
\special{pn 8}%
\special{ar 1870 1180 30 30  0.0000000 6.2831853}%
%
\special{pn 8}%
\special{pa 1930 1150}%
\special{pa 2210 850}%
\special{fp}%
\special{sh 1}%
\special{pa 2210 850}%
\special{pa 2150 886}%
\special{pa 2174 890}%
\special{pa 2180 912}%
\special{pa 2210 850}%
\special{fp}%
%
\special{pn 8}%
\special{ar 2240 800 32 32  0.0000000 6.2831853}%
%
\special{pn 8}%
\special{pa 660 800}%
\special{pa 1340 620}%
\special{fp}%
\special{sh 1}%
\special{pa 1340 620}%
\special{pa 1270 618}%
\special{pa 1288 634}%
\special{pa 1282 656}%
\special{pa 1340 620}%
\special{fp}%
%
\special{pn 8}%
\special{pa 1520 620}%
\special{pa 2170 780}%
\special{fp}%
\special{sh 1}%
\special{pa 2170 780}%
\special{pa 2110 746}%
\special{pa 2118 768}%
\special{pa 2100 784}%
\special{pa 2170 780}%
\special{fp}%
\put(21.9000,-7.3000){\makebox(0,0)[lb]{$0$}}%
\put(18.5000,-14.4000){\makebox(0,0)[lb]{$1$}}%
\put(7.8000,-14.4000){\makebox(0,0)[lb]{$n-3$}}%
\put(12.1000,-5.1000){\makebox(0,0)[lb]{$n-2$}}%
\put(4.2000,-6.8000){\makebox(0,0)[lb]{$n-1$}}%
\end{picture}%
\]
Let\\
$L_{0}=\{(0,0,0\cdots 0,0,0),(1,0,0\cdots 0,0,0),(1,1,0\cdots 0,0,0),\cdots, (1,1,1\cdots 1,0,0)\}$\\
$L_{1}=\{(1,0,0\cdots 0,1,0),(1,1,0\cdots 0,1,0),(1,1,1\cdots 0,1,0),\cdots, (1,1,1\cdots 1,1,0)\}$\\
$L_{2}=\{(2,1,0\cdots 0,1,0),(2,1,1\cdots 0,1,0),\cdots, (2,1,1\cdots 1,1,0)\}$\\
$L_{3}=\{(2,2,1,0\cdots 0,1,0),\cdots, (2,2,1\cdots 1,1,0)\}$

$\vdots$\\
$L_{n-3}=\{(2,2,2\cdots 2,1,1,0)\}$

and
\[T(a,b)=\left\{\begin{array}{ll}
\text{b-th\ elment\ of}\ L_{a} &0\leq a\leq n-3,\ 1\leq b\leq n-1-a \\ 
T(b-n+a,n-b)+(1,\cdots 1) & 0\leq a\leq n-3,\ n-1-a<b\leq n-1 \\ 
T(x,b)+(2r,2r,\cdots,2r) & a=x+(n-2)r\ (0\leq x<n-2),1\leq b\leq n-1.\  
\end{array}\right. \] 
Now we get
\[(\amalg L_{a},\leq^{op})\simeq (\mathrm{lk}_{\mathrm{p}}(P(n-1)),\leq)\]
and
\[(L(Q),\leq^{op})=(\{T(a,b)\mid a\in \Z_{\geq 0}, 1\leq b\leq n-1\},\leq^{op})
\simeq (\Z_{\geq 0}\times \{1,\cdots, n-1\},\leq^{op}). \]
In particular $\overrightarrow{\mathcal{T}}_{\mathrm{p}}(Q)=\Z_{\geq 0}\overrightarrow{A}_{n-1}$. 
So, in the case $n=4$, $\overrightarrow{\mathcal{T}}_{\mathrm{p}}(Q)$ is given by following 
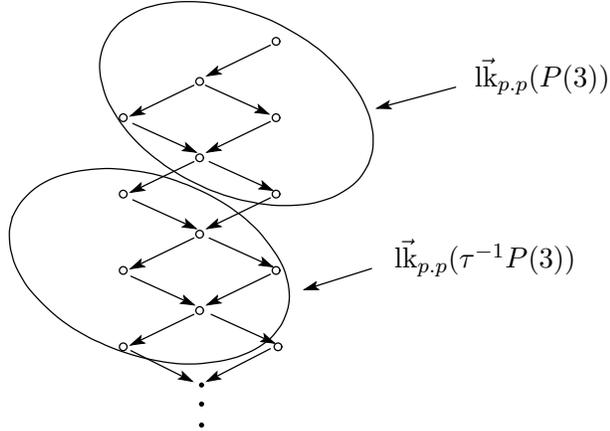
\[
\unitlength 0.1in
\begin{picture}( 24.3700, 22.1900)(  4.0300,-26.1000)
%
\special{pn 8}%
\special{ar 1800 600 20 20  0.0000000 6.2831853}%
%
\special{pn 8}%
\special{pa 1770 620}%
\special{pa 1440 780}%
\special{fp}%
\special{sh 1}%
\special{pa 1440 780}%
\special{pa 1510 770}%
\special{pa 1488 758}%
\special{pa 1492 734}%
\special{pa 1440 780}%
\special{fp}%
%
\special{pn 8}%
\special{ar 1400 810 20 20  0.0000000 6.2831853}%
%
\special{pn 8}%
\special{pa 1370 830}%
\special{pa 1040 990}%
\special{fp}%
\special{sh 1}%
\special{pa 1040 990}%
\special{pa 1110 980}%
\special{pa 1088 968}%
\special{pa 1092 944}%
\special{pa 1040 990}%
\special{fp}%
%
\special{pn 8}%
\special{ar 1000 1000 20 20  0.0000000 6.2831853}%
%
\special{pn 8}%
\special{ar 1800 1000 20 20  0.0000000 6.2831853}%
%
\special{pn 8}%
\special{pa 1770 1020}%
\special{pa 1440 1180}%
\special{fp}%
\special{sh 1}%
\special{pa 1440 1180}%
\special{pa 1510 1170}%
\special{pa 1488 1158}%
\special{pa 1492 1134}%
\special{pa 1440 1180}%
\special{fp}%
%
\special{pn 8}%
\special{ar 1400 1210 20 20  0.0000000 6.2831853}%
%
\special{pn 8}%
\special{pa 1370 1230}%
\special{pa 1040 1390}%
\special{fp}%
\special{sh 1}%
\special{pa 1040 1390}%
\special{pa 1110 1380}%
\special{pa 1088 1368}%
\special{pa 1092 1344}%
\special{pa 1040 1390}%
\special{fp}%
%
\special{pn 8}%
\special{ar 1000 1400 20 20  0.0000000 6.2831853}%
%
\special{pn 8}%
\special{ar 1800 1400 20 20  0.0000000 6.2831853}%
%
\special{pn 8}%
\special{pa 1770 1420}%
\special{pa 1440 1580}%
\special{fp}%
\special{sh 1}%
\special{pa 1440 1580}%
\special{pa 1510 1570}%
\special{pa 1488 1558}%
\special{pa 1492 1534}%
\special{pa 1440 1580}%
\special{fp}%
%
\special{pn 8}%
\special{ar 1400 1610 20 20  0.0000000 6.2831853}%
%
\special{pn 8}%
\special{pa 1370 1630}%
\special{pa 1040 1790}%
\special{fp}%
\special{sh 1}%
\special{pa 1040 1790}%
\special{pa 1110 1780}%
\special{pa 1088 1768}%
\special{pa 1092 1744}%
\special{pa 1040 1790}%
\special{fp}%
%
\special{pn 8}%
\special{ar 1000 1800 20 20  0.0000000 6.2831853}%
%
\special{pn 8}%
\special{ar 1800 1800 20 20  0.0000000 6.2831853}%
%
\special{pn 8}%
\special{pa 1770 1820}%
\special{pa 1440 1980}%
\special{fp}%
\special{sh 1}%
\special{pa 1440 1980}%
\special{pa 1510 1970}%
\special{pa 1488 1958}%
\special{pa 1492 1934}%
\special{pa 1440 1980}%
\special{fp}%
%
\special{pn 8}%
\special{ar 1400 2010 20 20  0.0000000 6.2831853}%
%
\special{pn 8}%
\special{pa 1370 2030}%
\special{pa 1040 2190}%
\special{fp}%
\special{sh 1}%
\special{pa 1040 2190}%
\special{pa 1110 2180}%
\special{pa 1088 2168}%
\special{pa 1092 2144}%
\special{pa 1040 2190}%
\special{fp}%
%
\special{pn 8}%
\special{ar 1000 2200 20 20  0.0000000 6.2831853}%
%
\special{pn 8}%
\special{pa 1430 830}%
\special{pa 1760 980}%
\special{fp}%
\special{sh 1}%
\special{pa 1760 980}%
\special{pa 1708 934}%
\special{pa 1712 958}%
\special{pa 1692 972}%
\special{pa 1760 980}%
\special{fp}%
%
\special{pn 8}%
\special{pa 1050 1030}%
\special{pa 1380 1180}%
\special{fp}%
\special{sh 1}%
\special{pa 1380 1180}%
\special{pa 1328 1134}%
\special{pa 1332 1158}%
\special{pa 1312 1172}%
\special{pa 1380 1180}%
\special{fp}%
%
\special{pn 8}%
\special{pa 1450 1230}%
\special{pa 1780 1380}%
\special{fp}%
\special{sh 1}%
\special{pa 1780 1380}%
\special{pa 1728 1334}%
\special{pa 1732 1358}%
\special{pa 1712 1372}%
\special{pa 1780 1380}%
\special{fp}%
%
\special{pn 8}%
\special{pa 1050 1430}%
\special{pa 1380 1580}%
\special{fp}%
\special{sh 1}%
\special{pa 1380 1580}%
\special{pa 1328 1534}%
\special{pa 1332 1558}%
\special{pa 1312 1572}%
\special{pa 1380 1580}%
\special{fp}%
%
\special{pn 8}%
\special{pa 1450 1630}%
\special{pa 1780 1780}%
\special{fp}%
\special{sh 1}%
\special{pa 1780 1780}%
\special{pa 1728 1734}%
\special{pa 1732 1758}%
\special{pa 1712 1772}%
\special{pa 1780 1780}%
\special{fp}%
%
\special{pn 8}%
\special{pa 1040 1830}%
\special{pa 1370 1980}%
\special{fp}%
\special{sh 1}%
\special{pa 1370 1980}%
\special{pa 1318 1934}%
\special{pa 1322 1958}%
\special{pa 1302 1972}%
\special{pa 1370 1980}%
\special{fp}%
%
\special{pn 8}%
\special{pa 1460 2030}%
\special{pa 1790 2180}%
\special{fp}%
\special{sh 1}%
\special{pa 1790 2180}%
\special{pa 1738 2134}%
\special{pa 1742 2158}%
\special{pa 1722 2172}%
\special{pa 1790 2180}%
\special{fp}%
%
\special{pn 8}%
\special{ar 1810 2200 20 20  0.0000000 6.2831853}%
%
\special{pn 8}%
\special{pa 1040 2220}%
\special{pa 1360 2380}%
\special{fp}%
\special{sh 1}%
\special{pa 1360 2380}%
\special{pa 1310 2332}%
\special{pa 1312 2356}%
\special{pa 1292 2368}%
\special{pa 1360 2380}%
\special{fp}%
%
\special{pn 8}%
\special{pa 1770 2210}%
\special{pa 1440 2380}%
\special{fp}%
\special{sh 1}%
\special{pa 1440 2380}%
\special{pa 1508 2368}%
\special{pa 1488 2356}%
\special{pa 1490 2332}%
\special{pa 1440 2380}%
\special{fp}%
%
\special{pn 8}%
\special{sh 1}%
\special{ar 1410 2400 10 10 0  6.28318530717959E+0000}%
\special{sh 1}%
\special{ar 1410 2500 10 10 0  6.28318530717959E+0000}%
\special{sh 1}%
\special{ar 1410 2610 10 10 0  6.28318530717959E+0000}%
%
\special{pn 8}%
\special{pa 2280 1250}%
\special{pa 2264 1278}%
\special{pa 2246 1306}%
\special{pa 2226 1330}%
\special{pa 2202 1352}%
\special{pa 2178 1370}%
\special{pa 2150 1388}%
\special{pa 2122 1404}%
\special{pa 2094 1418}%
\special{pa 2064 1428}%
\special{pa 2032 1438}%
\special{pa 2002 1446}%
\special{pa 1970 1452}%
\special{pa 1938 1456}%
\special{pa 1906 1458}%
\special{pa 1874 1460}%
\special{pa 1842 1460}%
\special{pa 1810 1460}%
\special{pa 1778 1458}%
\special{pa 1746 1454}%
\special{pa 1714 1450}%
\special{pa 1684 1444}%
\special{pa 1652 1438}%
\special{pa 1620 1432}%
\special{pa 1590 1424}%
\special{pa 1558 1414}%
\special{pa 1528 1406}%
\special{pa 1498 1394}%
\special{pa 1468 1384}%
\special{pa 1438 1372}%
\special{pa 1410 1358}%
\special{pa 1380 1344}%
\special{pa 1352 1330}%
\special{pa 1324 1316}%
\special{pa 1296 1300}%
\special{pa 1268 1282}%
\special{pa 1242 1266}%
\special{pa 1214 1248}%
\special{pa 1190 1228}%
\special{pa 1164 1210}%
\special{pa 1138 1188}%
\special{pa 1114 1168}%
\special{pa 1090 1146}%
\special{pa 1068 1124}%
\special{pa 1046 1100}%
\special{pa 1024 1076}%
\special{pa 1006 1052}%
\special{pa 986 1026}%
\special{pa 968 998}%
\special{pa 952 972}%
\special{pa 936 944}%
\special{pa 920 916}%
\special{pa 908 886}%
\special{pa 898 856}%
\special{pa 890 824}%
\special{pa 882 794}%
\special{pa 878 762}%
\special{pa 876 730}%
\special{pa 878 698}%
\special{pa 884 666}%
\special{pa 892 636}%
\special{pa 904 606}%
\special{pa 920 578}%
\special{pa 936 550}%
\special{pa 958 526}%
\special{pa 980 502}%
\special{pa 1006 482}%
\special{pa 1032 466}%
\special{pa 1060 450}%
\special{pa 1088 436}%
\special{pa 1118 424}%
\special{pa 1150 416}%
\special{pa 1180 408}%
\special{pa 1212 402}%
\special{pa 1244 398}%
\special{pa 1276 394}%
\special{pa 1308 392}%
\special{pa 1340 392}%
\special{pa 1372 392}%
\special{pa 1404 394}%
\special{pa 1436 396}%
\special{pa 1468 400}%
\special{pa 1498 406}%
\special{pa 1530 412}%
\special{pa 1562 418}%
\special{pa 1592 426}%
\special{pa 1624 436}%
\special{pa 1654 446}%
\special{pa 1684 456}%
\special{pa 1714 466}%
\special{pa 1744 478}%
\special{pa 1774 490}%
\special{pa 1802 504}%
\special{pa 1830 518}%
\special{pa 1858 534}%
\special{pa 1886 550}%
\special{pa 1914 566}%
\special{pa 1942 584}%
\special{pa 1968 600}%
\special{pa 1994 620}%
\special{pa 2020 638}%
\special{pa 2044 658}%
\special{pa 2068 680}%
\special{pa 2092 702}%
\special{pa 2114 724}%
\special{pa 2136 748}%
\special{pa 2158 772}%
\special{pa 2178 796}%
\special{pa 2198 822}%
\special{pa 2216 848}%
\special{pa 2234 874}%
\special{pa 2250 902}%
\special{pa 2264 932}%
\special{pa 2276 960}%
\special{pa 2288 990}%
\special{pa 2296 1022}%
\special{pa 2302 1052}%
\special{pa 2308 1084}%
\special{pa 2308 1116}%
\special{pa 2308 1148}%
\special{pa 2304 1180}%
\special{pa 2294 1212}%
\special{pa 2284 1242}%
\special{pa 2280 1250}%
\special{sp}%
%
\special{pn 8}%
\special{pa 2740 840}%
\special{pa 2330 950}%
\special{fp}%
\special{sh 1}%
\special{pa 2330 950}%
\special{pa 2400 952}%
\special{pa 2382 936}%
\special{pa 2390 914}%
\special{pa 2330 950}%
\special{fp}%
\put(28.4000,-8.8000){\makebox(0,0)[lb]{$\vec{\mathrm{lk}}_{p.p}(P(3))$}}%
%
\special{pn 8}%
\special{pa 420 1520}%
\special{pa 434 1492}%
\special{pa 450 1464}%
\special{pa 468 1438}%
\special{pa 490 1414}%
\special{pa 512 1392}%
\special{pa 538 1370}%
\special{pa 564 1352}%
\special{pa 592 1336}%
\special{pa 620 1322}%
\special{pa 650 1310}%
\special{pa 680 1300}%
\special{pa 710 1292}%
\special{pa 742 1284}%
\special{pa 774 1278}%
\special{pa 806 1274}%
\special{pa 838 1270}%
\special{pa 870 1268}%
\special{pa 902 1266}%
\special{pa 934 1266}%
\special{pa 964 1268}%
\special{pa 996 1270}%
\special{pa 1028 1274}%
\special{pa 1060 1278}%
\special{pa 1092 1282}%
\special{pa 1124 1288}%
\special{pa 1156 1294}%
\special{pa 1186 1302}%
\special{pa 1216 1310}%
\special{pa 1248 1320}%
\special{pa 1278 1330}%
\special{pa 1308 1342}%
\special{pa 1338 1352}%
\special{pa 1368 1366}%
\special{pa 1396 1378}%
\special{pa 1426 1392}%
\special{pa 1454 1406}%
\special{pa 1482 1422}%
\special{pa 1510 1438}%
\special{pa 1536 1454}%
\special{pa 1564 1472}%
\special{pa 1590 1492}%
\special{pa 1616 1510}%
\special{pa 1640 1530}%
\special{pa 1664 1552}%
\special{pa 1688 1574}%
\special{pa 1710 1596}%
\special{pa 1730 1620}%
\special{pa 1750 1646}%
\special{pa 1770 1672}%
\special{pa 1788 1698}%
\special{pa 1806 1724}%
\special{pa 1822 1752}%
\special{pa 1836 1780}%
\special{pa 1848 1810}%
\special{pa 1856 1842}%
\special{pa 1864 1872}%
\special{pa 1868 1904}%
\special{pa 1870 1936}%
\special{pa 1868 1968}%
\special{pa 1862 2000}%
\special{pa 1852 2030}%
\special{pa 1840 2060}%
\special{pa 1826 2088}%
\special{pa 1806 2114}%
\special{pa 1786 2140}%
\special{pa 1764 2162}%
\special{pa 1738 2182}%
\special{pa 1712 2200}%
\special{pa 1684 2216}%
\special{pa 1656 2230}%
\special{pa 1626 2242}%
\special{pa 1596 2254}%
\special{pa 1566 2262}%
\special{pa 1534 2270}%
\special{pa 1502 2276}%
\special{pa 1470 2280}%
\special{pa 1440 2284}%
\special{pa 1408 2288}%
\special{pa 1376 2290}%
\special{pa 1344 2290}%
\special{pa 1312 2290}%
\special{pa 1280 2288}%
\special{pa 1248 2284}%
\special{pa 1216 2280}%
\special{pa 1184 2274}%
\special{pa 1152 2268}%
\special{pa 1122 2262}%
\special{pa 1090 2254}%
\special{pa 1060 2246}%
\special{pa 1028 2236}%
\special{pa 998 2226}%
\special{pa 968 2216}%
\special{pa 938 2204}%
\special{pa 908 2192}%
\special{pa 880 2178}%
\special{pa 850 2166}%
\special{pa 822 2152}%
\special{pa 794 2136}%
\special{pa 766 2120}%
\special{pa 738 2104}%
\special{pa 712 2086}%
\special{pa 686 2068}%
\special{pa 660 2048}%
\special{pa 636 2028}%
\special{pa 612 2006}%
\special{pa 588 1984}%
\special{pa 566 1962}%
\special{pa 544 1938}%
\special{pa 522 1914}%
\special{pa 502 1890}%
\special{pa 484 1864}%
\special{pa 468 1836}%
\special{pa 452 1808}%
\special{pa 438 1780}%
\special{pa 426 1750}%
\special{pa 418 1718}%
\special{pa 410 1688}%
\special{pa 406 1656}%
\special{pa 404 1624}%
\special{pa 404 1592}%
\special{pa 410 1560}%
\special{pa 418 1528}%
\special{pa 420 1520}%
\special{sp}%
%
\special{pn 8}%
\special{pa 2300 1790}%
\special{pa 1950 1900}%
\special{fp}%
\special{sh 1}%
\special{pa 1950 1900}%
\special{pa 2020 1900}%
\special{pa 2002 1884}%
\special{pa 2008 1862}%
\special{pa 1950 1900}%
\special{fp}%
\put(24.2000,-18.3000){\makebox(0,0)[lb]{$\vec{\mathrm{lk}}_{p.p}(\tau^{-1}P(3))$}}%
\end{picture}%
\]

(2):Consider the following quiver
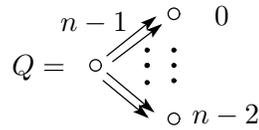
\[
\unitlength 0.1in
\begin{picture}( 10.0500,  6.5500)(  4.3000,-12.9000)
\put(6.1000,-9.9000){\makebox(0,0){$Q=$}}%
%
\special{pn 8}%
\special{ar 910 980 30 30  0.0000000 6.2831853}%
%
\special{pn 8}%
\special{pa 1000 960}%
\special{pa 1250 770}%
\special{fp}%
\special{sh 1}%
\special{pa 1250 770}%
\special{pa 1186 794}%
\special{pa 1208 802}%
\special{pa 1210 826}%
\special{pa 1250 770}%
\special{fp}%
%
\special{pn 8}%
\special{pa 970 930}%
\special{pa 1220 730}%
\special{fp}%
\special{sh 1}%
\special{pa 1220 730}%
\special{pa 1156 756}%
\special{pa 1178 764}%
\special{pa 1180 788}%
\special{pa 1220 730}%
\special{fp}%
%
\special{pn 8}%
\special{sh 1}%
\special{ar 1180 900 10 10 0  6.28318530717959E+0000}%
\special{sh 1}%
\special{ar 1180 980 10 10 0  6.28318530717959E+0000}%
\special{sh 1}%
\special{ar 1180 980 10 10 0  6.28318530717959E+0000}%
\special{sh 1}%
\special{ar 1180 980 10 10 0  6.28318530717959E+0000}%
%
\special{pn 8}%
\special{sh 1}%
\special{ar 1180 1070 10 10 0  6.28318530717959E+0000}%
%
\special{pn 8}%
\special{pa 970 1020}%
\special{pa 1210 1220}%
\special{fp}%
\special{sh 1}%
\special{pa 1210 1220}%
\special{pa 1172 1162}%
\special{pa 1170 1186}%
\special{pa 1146 1194}%
\special{pa 1210 1220}%
\special{fp}%
%
\special{pn 8}%
\special{pa 950 1060}%
\special{pa 1180 1250}%
\special{fp}%
\special{sh 1}%
\special{pa 1180 1250}%
\special{pa 1142 1192}%
\special{pa 1140 1216}%
\special{pa 1116 1224}%
\special{pa 1180 1250}%
\special{fp}%
%
\special{pn 8}%
\special{ar 1320 710 30 30  0.0000000 6.2831853}%
%
\special{pn 8}%
\special{sh 1}%
\special{ar 1320 900 10 10 0  6.28318530717959E+0000}%
\special{sh 1}%
\special{ar 1320 980 10 10 0  6.28318530717959E+0000}%
\special{sh 1}%
\special{ar 1320 1060 10 10 0  6.28318530717959E+0000}%
%
\special{pn 8}%
\special{ar 1320 1260 30 30  0.0000000 6.2831853}%
\put(15.7000,-7.2000){\makebox(0,0){$0$}}%
\put(15.9000,-12.4000){\makebox(0,0){$n-2$}}%
\put(9.0000,-7.6000){\makebox(0,0){$n-1$}}%
\end{picture}%
\]

It is easy to see that
\[(\{0,1\}^{n-1},\leq^{op})\simeq (\mathrm{lk}_{\mathrm{p}}(P(n-1)),\leq)\]
and so an underlying graph of $\overrightarrow{\mathrm{lk}}_{\mathrm{p}}(P(n-1))$ is 
isomorphic to (n-1)-dimensional cube. In the case $n=4$, $\overrightarrow{\mathcal{T}}_{\mathrm{p}}(Q)$ is given by following,
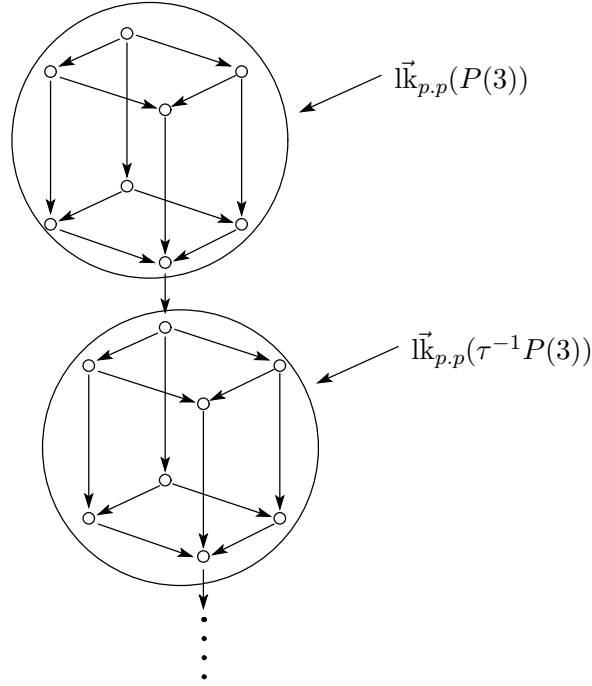
\[
\unitlength 0.1in
\begin{picture}( 20.9100, 35.3400)(  5.9900,-37.7300)
%
\special{pn 8}%
\special{ar 1200 400 30 30  0.0000000 6.2831853}%
%
\special{pn 8}%
\special{pa 1160 430}%
\special{pa 850 570}%
\special{fp}%
\special{sh 1}%
\special{pa 850 570}%
\special{pa 920 562}%
\special{pa 900 548}%
\special{pa 904 524}%
\special{pa 850 570}%
\special{fp}%
%
\special{pn 8}%
\special{ar 800 600 30 30  0.0000000 6.2831853}%
%
\special{pn 8}%
\special{ar 1800 600 30 30  0.0000000 6.2831853}%
%
\special{pn 8}%
\special{ar 1400 800 30 30  0.0000000 6.2831853}%
%
\special{pn 8}%
\special{ar 1200 1200 30 30  0.0000000 6.2831853}%
%
\special{pn 8}%
\special{ar 800 1400 30 30  0.0000000 6.2831853}%
%
\special{pn 8}%
\special{ar 1400 1600 30 30  0.0000000 6.2831853}%
%
\special{pn 8}%
\special{ar 1800 1400 30 30  0.0000000 6.2831853}%
%
\special{pn 8}%
\special{pa 850 630}%
\special{pa 1350 790}%
\special{fp}%
\special{sh 1}%
\special{pa 1350 790}%
\special{pa 1294 752}%
\special{pa 1300 774}%
\special{pa 1280 790}%
\special{pa 1350 790}%
\special{fp}%
%
\special{pn 8}%
\special{pa 1250 410}%
\special{pa 1740 570}%
\special{fp}%
\special{sh 1}%
\special{pa 1740 570}%
\special{pa 1684 530}%
\special{pa 1690 554}%
\special{pa 1670 568}%
\special{pa 1740 570}%
\special{fp}%
%
\special{pn 8}%
\special{pa 1770 620}%
\special{pa 1440 780}%
\special{fp}%
\special{sh 1}%
\special{pa 1440 780}%
\special{pa 1510 770}%
\special{pa 1488 758}%
\special{pa 1492 734}%
\special{pa 1440 780}%
\special{fp}%
%
\special{pn 8}%
\special{pa 1200 450}%
\special{pa 1200 1150}%
\special{fp}%
\special{sh 1}%
\special{pa 1200 1150}%
\special{pa 1220 1084}%
\special{pa 1200 1098}%
\special{pa 1180 1084}%
\special{pa 1200 1150}%
\special{fp}%
%
\special{pn 8}%
\special{pa 800 650}%
\special{pa 800 1340}%
\special{fp}%
\special{sh 1}%
\special{pa 800 1340}%
\special{pa 820 1274}%
\special{pa 800 1288}%
\special{pa 780 1274}%
\special{pa 800 1340}%
\special{fp}%
%
\special{pn 8}%
\special{pa 1800 640}%
\special{pa 1800 1350}%
\special{fp}%
\special{sh 1}%
\special{pa 1800 1350}%
\special{pa 1820 1284}%
\special{pa 1800 1298}%
\special{pa 1780 1284}%
\special{pa 1800 1350}%
\special{fp}%
%
\special{pn 8}%
\special{pa 1240 1220}%
\special{pa 1750 1390}%
\special{fp}%
\special{sh 1}%
\special{pa 1750 1390}%
\special{pa 1694 1350}%
\special{pa 1700 1374}%
\special{pa 1680 1388}%
\special{pa 1750 1390}%
\special{fp}%
%
\special{pn 8}%
\special{pa 1170 1230}%
\special{pa 850 1380}%
\special{fp}%
\special{sh 1}%
\special{pa 850 1380}%
\special{pa 920 1370}%
\special{pa 898 1358}%
\special{pa 902 1334}%
\special{pa 850 1380}%
\special{fp}%
%
\special{pn 8}%
\special{pa 850 1430}%
\special{pa 1340 1600}%
\special{fp}%
\special{sh 1}%
\special{pa 1340 1600}%
\special{pa 1284 1560}%
\special{pa 1290 1584}%
\special{pa 1270 1598}%
\special{pa 1340 1600}%
\special{fp}%
%
\special{pn 8}%
\special{pa 1770 1430}%
\special{pa 1450 1590}%
\special{fp}%
\special{sh 1}%
\special{pa 1450 1590}%
\special{pa 1520 1578}%
\special{pa 1498 1566}%
\special{pa 1502 1542}%
\special{pa 1450 1590}%
\special{fp}%
%
\special{pn 8}%
\special{pa 1400 840}%
\special{pa 1400 1550}%
\special{fp}%
\special{sh 1}%
\special{pa 1400 1550}%
\special{pa 1420 1484}%
\special{pa 1400 1498}%
\special{pa 1380 1484}%
\special{pa 1400 1550}%
\special{fp}%
%
\special{pn 8}%
\special{ar 1400 1940 30 30  0.0000000 6.2831853}%
%
\special{pn 8}%
\special{pa 1360 1970}%
\special{pa 1050 2110}%
\special{fp}%
\special{sh 1}%
\special{pa 1050 2110}%
\special{pa 1120 2102}%
\special{pa 1100 2088}%
\special{pa 1104 2064}%
\special{pa 1050 2110}%
\special{fp}%
%
\special{pn 8}%
\special{ar 1000 2140 30 30  0.0000000 6.2831853}%
%
\special{pn 8}%
\special{ar 2000 2140 30 30  0.0000000 6.2831853}%
%
\special{pn 8}%
\special{ar 1600 2340 30 30  0.0000000 6.2831853}%
%
\special{pn 8}%
\special{ar 1400 2740 30 30  0.0000000 6.2831853}%
%
\special{pn 8}%
\special{ar 1000 2940 30 30  0.0000000 6.2831853}%
%
\special{pn 8}%
\special{ar 1600 3140 30 30  0.0000000 6.2831853}%
%
\special{pn 8}%
\special{ar 2000 2940 30 30  0.0000000 6.2831853}%
%
\special{pn 8}%
\special{pa 1050 2170}%
\special{pa 1550 2330}%
\special{fp}%
\special{sh 1}%
\special{pa 1550 2330}%
\special{pa 1494 2292}%
\special{pa 1500 2314}%
\special{pa 1480 2330}%
\special{pa 1550 2330}%
\special{fp}%
%
\special{pn 8}%
\special{pa 1450 1950}%
\special{pa 1940 2110}%
\special{fp}%
\special{sh 1}%
\special{pa 1940 2110}%
\special{pa 1884 2070}%
\special{pa 1890 2094}%
\special{pa 1870 2108}%
\special{pa 1940 2110}%
\special{fp}%
%
\special{pn 8}%
\special{pa 1970 2160}%
\special{pa 1640 2320}%
\special{fp}%
\special{sh 1}%
\special{pa 1640 2320}%
\special{pa 1710 2310}%
\special{pa 1688 2298}%
\special{pa 1692 2274}%
\special{pa 1640 2320}%
\special{fp}%
%
\special{pn 8}%
\special{pa 1400 1990}%
\special{pa 1400 2690}%
\special{fp}%
\special{sh 1}%
\special{pa 1400 2690}%
\special{pa 1420 2624}%
\special{pa 1400 2638}%
\special{pa 1380 2624}%
\special{pa 1400 2690}%
\special{fp}%
%
\special{pn 8}%
\special{pa 1000 2190}%
\special{pa 1000 2880}%
\special{fp}%
\special{sh 1}%
\special{pa 1000 2880}%
\special{pa 1020 2814}%
\special{pa 1000 2828}%
\special{pa 980 2814}%
\special{pa 1000 2880}%
\special{fp}%
%
\special{pn 8}%
\special{pa 2000 2180}%
\special{pa 2000 2890}%
\special{fp}%
\special{sh 1}%
\special{pa 2000 2890}%
\special{pa 2020 2824}%
\special{pa 2000 2838}%
\special{pa 1980 2824}%
\special{pa 2000 2890}%
\special{fp}%
%
\special{pn 8}%
\special{pa 1440 2760}%
\special{pa 1950 2930}%
\special{fp}%
\special{sh 1}%
\special{pa 1950 2930}%
\special{pa 1894 2890}%
\special{pa 1900 2914}%
\special{pa 1880 2928}%
\special{pa 1950 2930}%
\special{fp}%
%
\special{pn 8}%
\special{pa 1370 2770}%
\special{pa 1050 2920}%
\special{fp}%
\special{sh 1}%
\special{pa 1050 2920}%
\special{pa 1120 2910}%
\special{pa 1098 2898}%
\special{pa 1102 2874}%
\special{pa 1050 2920}%
\special{fp}%
%
\special{pn 8}%
\special{pa 1050 2970}%
\special{pa 1540 3140}%
\special{fp}%
\special{sh 1}%
\special{pa 1540 3140}%
\special{pa 1484 3100}%
\special{pa 1490 3124}%
\special{pa 1470 3138}%
\special{pa 1540 3140}%
\special{fp}%
%
\special{pn 8}%
\special{pa 1970 2970}%
\special{pa 1650 3130}%
\special{fp}%
\special{sh 1}%
\special{pa 1650 3130}%
\special{pa 1720 3118}%
\special{pa 1698 3106}%
\special{pa 1702 3082}%
\special{pa 1650 3130}%
\special{fp}%
%
\special{pn 8}%
\special{pa 1600 2380}%
\special{pa 1600 3090}%
\special{fp}%
\special{sh 1}%
\special{pa 1600 3090}%
\special{pa 1620 3024}%
\special{pa 1600 3038}%
\special{pa 1580 3024}%
\special{pa 1600 3090}%
\special{fp}%
%
\special{pn 8}%
\special{sh 1}%
\special{ar 1604 3470 10 10 0  6.28318530717959E+0000}%
\special{sh 1}%
\special{ar 1604 3470 10 10 0  6.28318530717959E+0000}%
\special{sh 1}%
\special{ar 1604 3470 10 10 0  6.28318530717959E+0000}%
\special{sh 1}%
\special{ar 1604 3470 10 10 0  6.28318530717959E+0000}%
%
\special{pn 8}%
\special{sh 1}%
\special{ar 1604 3570 10 10 0  6.28318530717959E+0000}%
%
\special{pn 8}%
\special{sh 1}%
\special{ar 1604 3670 10 10 0  6.28318530717959E+0000}%
%
\special{pn 8}%
\special{sh 1}%
\special{ar 1604 3770 10 10 0  6.28318530717959E+0000}%
%
\special{pn 8}%
\special{pa 1400 1660}%
\special{pa 1400 1860}%
\special{fp}%
\special{sh 1}%
\special{pa 1400 1860}%
\special{pa 1420 1794}%
\special{pa 1400 1808}%
\special{pa 1380 1794}%
\special{pa 1400 1860}%
\special{fp}%
%
\special{pn 8}%
\special{pa 1600 3210}%
\special{pa 1600 3410}%
\special{fp}%
\special{sh 1}%
\special{pa 1600 3410}%
\special{pa 1620 3344}%
\special{pa 1600 3358}%
\special{pa 1580 3344}%
\special{pa 1600 3410}%
\special{fp}%
%
\special{pn 8}%
\special{ar 1320 960 722 722  0.0000000 6.2831853}%
%
\special{pn 8}%
\special{ar 1480 2570 722 722  0.0000000 6.2831853}%
%
\special{pn 8}%
\special{pa 2530 620}%
\special{pa 2110 810}%
\special{fp}%
\special{sh 1}%
\special{pa 2110 810}%
\special{pa 2180 802}%
\special{pa 2160 788}%
\special{pa 2162 764}%
\special{pa 2110 810}%
\special{fp}%
\put(26.0000,-7.3000){\makebox(0,0)[lb]{$\vec{\mathrm{lk}}_{p.p}(P(3))$}}%
%
\special{pn 8}%
\special{pa 2620 2040}%
\special{pa 2200 2230}%
\special{fp}%
\special{sh 1}%
\special{pa 2200 2230}%
\special{pa 2270 2222}%
\special{pa 2250 2208}%
\special{pa 2252 2184}%
\special{pa 2200 2230}%
\special{fp}%
\put(26.9000,-21.5000){\makebox(0,0)[lb]{$\vec{\mathrm{lk}}_{p.p}(\tau^{-1}P(3))$}}%
\end{picture}%
\]
 
\end{exmp}

\section{An application}
In this section we consider a quiver $\overrightarrow{\mathcal{T}}_{\mathrm{p}}(Q)$ for 
$Q$ satisfying condition (b) of Theorem\;\ref{t} and $l(Q):=\mathrm{max}\{l_{Q}(x,y)\mid x,y\in Q_{0}\}\leq 1$.

Denote by $\mathcal{Q}$ the set of finite connected quivers without loops and cycles. For any quiver $Q\in \mathcal{Q}$ define a new quiver
$Q^{\circ}$ by adding new edges $x\rightarrow y$ for any source $x$ which is not sink and 
sink $y$ which is not source. 
Then let $\mathcal{A}:=\{Q\in \mathcal{Q}\mid Q\ \mathrm{has\ a\ unique\ source}\}$, 
$\mathcal{B}:=\{Q\in \mathcal{A}\mid Q\ \mathrm{has\ a\ unique\ sink }\}$ 
and $\mathcal{A}^{\circ}:=\{Q^{\circ}\mid Q\in \mathcal{A}\}$. Note that $\mathcal{A}^{\circ}=\{Q\in \mathcal{A}
 \mid l(Q)\leq 1\}$.
 
\begin{defn}
We define the maps $\overrightarrow{\amalg}:\mathcal{Q}\times \mathcal{Q}\rightarrow \mathcal{Q}$, 
$\phi:\mathcal{A}^{\circ}\times \mathcal{A}^{\circ}\rightarrow \mathcal{A}^{\circ}$, $\psi:\mathcal{A}^{\circ} \rightarrow \mathcal{B}$ and 
$\Psi:\mathcal{A}^{\circ}\times \mathcal{A}^{\circ}\rightarrow \mathcal{B}\times \mathcal{B}$
as follows,\\
$(1)$\;$(Q\overrightarrow{\amalg}Q^{'})_{0}:=Q_{0}\amalg Q^{'}_{0}$,

\ $\;(Q\overrightarrow{\amalg}Q^{'})_{1}:=Q_{1}\amalg Q^{'}_{1}\amalg \{y\rightarrow x^{'}\mid y:\mathrm{source\ of\ Q},\ x^{'}:\mathrm{sink\ of\ Q^{'}} \}$,\\
$(2)$\;$\phi(Q,Q^{'}):=(Q\overrightarrow{\amalg}Q^{'})^{\circ},$\\
$(3)$\;$\psi(Q):=\overrightarrow{\mathrm{lk}}_{\mathrm{p}}(P_{Q})$, where $P_{Q}$ is an indecomposable projective module associated with a unique source,\\
$(4)$\;$\Psi(Q,Q^{'}):=(\psi(Q^{'}),\psi(Q)))$.

\end{defn}

\begin{prop}
\label{sp}
The following diagram is commutative,

\[
\unitlength 0.1in
\begin{picture}( 15.5500, 13.5000)(  6.1000,-15.9000)
\put(6.1000,-6.0000){\makebox(0,0)[lb]{$\mathcal{A}^{\circ}\times \mathcal{A}^{\circ}$}}%
\put(8.0000,-16.2000){\makebox(0,0)[lb]{$\mathcal{A}^{\circ}$}}%
\put(18.8000,-5.9000){\makebox(0,0)[lb]{$\mathcal{B}\times \mathcal{B}$}}%
\put(19.8000,-15.9000){\makebox(0,0)[lb]{$\mathcal{B}$}}%
\put(14.6000,-4.1000){\makebox(0,0)[lb]{$\Psi$}}%
\put(13.8000,-17.6000){\makebox(0,0)[lb]{$\psi$}}%
\put(6.3000,-11.1000){\makebox(0,0)[lb]{$\phi$}}%
\put(21.6500,-11.1000){\makebox(0,0)[lb]{$\vec{\amalg}$}}%
%
\special{pn 8}%
\special{pa 970 670}%
\special{pa 970 1340}%
\special{fp}%
\special{sh 1}%
\special{pa 970 1340}%
\special{pa 990 1274}%
\special{pa 970 1288}%
\special{pa 950 1274}%
\special{pa 970 1340}%
\special{fp}%
%
\special{pn 8}%
\special{pa 2066 670}%
\special{pa 2066 1340}%
\special{fp}%
\special{sh 1}%
\special{pa 2066 1340}%
\special{pa 2086 1274}%
\special{pa 2066 1288}%
\special{pa 2046 1274}%
\special{pa 2066 1340}%
\special{fp}%
%
\special{pn 8}%
\special{pa 1160 530}%
\special{pa 1730 530}%
\special{fp}%
\special{sh 1}%
\special{pa 1730 530}%
\special{pa 1664 510}%
\special{pa 1678 530}%
\special{pa 1664 550}%
\special{pa 1730 530}%
\special{fp}%
%
\special{pn 8}%
\special{pa 1160 1530}%
\special{pa 1730 1530}%
\special{fp}%
\special{sh 1}%
\special{pa 1730 1530}%
\special{pa 1664 1510}%
\special{pa 1678 1530}%
\special{pa 1664 1550}%
\special{pa 1730 1530}%
\special{fp}%
\end{picture}%
  \]

\end{prop}
\begin{proof}
Let $Q(1),Q(2)\in \mathcal{A}^{\circ}$, $s(k)$ be a unique source of $Q(k)\;(k=1,2)$ and $T=\oplus_{i\in \phi(Q(1),Q(2))_{0}}\tau^{-r_{i}}P(i)$. Then Corollary\;\ref{c} shows $T\in\psi(\phi(Q(1),Q(2)))_{0}$ if and only if  
(i):$(r_{i})_{i\in Q(2)_{0}}\in L(Q(2))$ with $r_{s(2)}=0$ and $r_{j}=0\; (\forall j\in Q(1)_{0})$ or (ii):$(r_{j})_{j\in Q(1)_{0}}\in L(Q(1))$ with $r_{s(1)}=0$ and $r_{i}=1\;(\forall i\in Q(2)_{0}).$
Now it is easy to check that 
\[\psi(\phi(Q(1),Q(2)))=\psi(Q(2))\overrightarrow{\amalg}\psi (Q(1)).\]
\end{proof}

Now we define  a relation  $\leadsto $ on $\mathcal{A}^{\circ}$ as  follows,
\[Q\leadsto Q^{'}\stackrel{\mathrm{def}}{\Leftrightarrow} Q^{'}=Q\setminus \{\alpha\}\]
for some $\alpha \in Q_{1}$ satisfying the conditions $(1)$ or $(2)$;\\
$(1)\;$ $s(\alpha)$ is not source or $t(\alpha)$ is not sink and $\exists w\neq \alpha$ path from $s(\alpha)$ to $t(\alpha)$.\\
$(2)\;$ $s(\alpha)$ is a source, $t(\alpha)$ is a sink and there is at least three paths from $s(\alpha)$ to $t(\alpha)$ and at least two
arrows from $s(\alpha)$ to $t(\alpha)$. 

Let $\mathcal{S}:=\{Q\in \mathcal{A}^{\circ}\mid \mathrm{there\ is\ a\ no\ quiver}\ Q^{'}\ \mathrm{s.t.}\;Q\leadsto Q^{'}\}$. We can easyly see that if 
\[Q\leadsto \cdots \leadsto Q^{'}\in \mathcal{S},\ Q\leadsto \cdots \leadsto Q^{''}\in \mathcal{S},\]
then $Q^{'}=Q^{''}$ and in this case put $\pi (Q):=Q^{'}=Q^{''}$.
Now we define a equivalence relation $\sim $ on $\mathcal{A}^{\circ}$ as follows,
\[Q\sim Q^{'}\stackrel{\mathrm{def}}{\Leftrightarrow}\pi (Q)=\pi (Q^{'}).\]  
 
\begin{lem}
\label{sl1}
Let $Q, Q^{'}\in \mathcal{A}^{\circ}$. If $Q \sim Q^{'}$, then $\psi(Q)=\psi(Q)$. In particular we get a map 
\[\psi/\sim:\mathcal{A}^{\circ}/\sim\rightarrow \mathcal{B}.\] 

\end{lem}
\begin{proof}
Note that if $Q\leadsto Q^{'}$ then $l_{Q}(i,j)=l_{Q^{'}}(i,j)$ for any $i,j\in Q_{0}=Q^{'}_{0}$. In particular Corollary\;\ref{c} shows
$\psi(Q)=\psi (Q^{'}).$
\end{proof} 
\begin{lem}
\label{sl2}
Let $Q(1),Q(2),Q^{'}(1),Q^{'}(2)\in\mathcal{A}^{\circ}$. If $Q(i)\sim Q^{'}(i)\;(i=1,2)$, then $\phi(Q(1),Q(2))\sim\phi(Q^{'}(1),Q^{'}(2)).$ In particular we get a map
\[\phi/\sim:\mathcal{A}^{\circ}/\sim\times \mathcal{A}^{\circ}/\sim\rightarrow \mathcal{A}^{\circ}/\sim.\]
\end{lem}
\begin{proof}
Let $Q,Q^{'},Q^{''}\in \mathcal{A}^{\circ}$ with $Q\leadsto Q^{'}=Q\setminus \{\alpha\}$. By definition we get $\phi(Q,Q^{''})=\phi(Q,Q^{''})\setminus \{\alpha\}$. And now  
$\alpha$ satisfies the condition (1) of definition for  
$\phi(Q^{'},Q^{''})\leadsto\phi(Q^{'},Q{''})$. In particular we get
\[\phi(Q,Q^{''})\leadsto \phi(\pi(Q),Q^{''}).\]
Similarly we can see that 
\[\phi(Q,Q^{''})\leadsto \phi(Q,\pi(Q^{''})). \]
So we get $\phi(Q,Q^{''})\leadsto \phi(\pi(Q),\pi(Q^{''})).$
\end{proof}
Let $C^{n}$ be a Hasse-quiver of $(\{0,1\}^{n},\leq)$, and $Q\in \mathcal{A}^{\circ}$ with $Q_{0}=\{s,1,2,\cdots n-1\}$ where $s$ is a unique source of $Q$. Then a map 
\[\rho: P(s)\oplus(\oplus_{i=1}^{n-1}\tau^{-r_{i}}P(i))\mapsto (r(i))_{i}\] 
induces an injection
$\psi(Q)\rightarrow C^{n-1}$ as a quiver. So we identify $\psi(Q)$ as
a full sub-quiver of $C^{n-1}$. Note that $(0,\cdots,0 ),(1,\cdots 1)\in \psi (Q)$. Now for any $T\in C^{n-1}_{0}$ denote by
$T_{i}$ the $i$-th entry of $T$.
\begin{prop}
\label{sp1}
Let $Q\in \mathcal{A}^{\circ}$. Then $\psi(Q)=K(1)\overrightarrow{\amalg}K(2)$ 
for some quivers $K(1),K(2)\in \mathcal{B}$ 
if and only if $\exists (Q(1),Q(2))\in \mathcal{A}^{\circ}\times 
\mathcal{A}^{\circ}$\;s.t.\;$\psi(Q(i))=K(i)\;(i=1,2)$ 
 and $\phi(Q(2),Q(1))\sim Q$. 

\end{prop}
\begin{proof}
First assume that $\exists (Q(1),Q(2))\in \mathcal{A}^{\circ}\times 
\mathcal{A}^{\circ}$\;s.t.\;$\phi(Q(2),Q(1))\sim Q$ and put $\psi(Q(i))=K(i)\;(i=1,2).$
Then proposition\;\ref{sp}, Lemma\:\ref{sl1} and Lemma\;\ref{sl2} implies
 \[\psi(Q)=K(1)\overrightarrow{\amalg}K(2). \]

Next assume that $\psi(Q)=K(1)\overrightarrow{\amalg}K(2)$ 
for some quivers $K(1),K(2)\in \mathcal{B}.$
By above Lemma we can assume that $Q\in \mathcal{S}$.
Let $T$ be the unique minimal element of $K_{1}$ and $T^{'}$ be the 
unique maximal element of $K_{2}$. Then $\exists 1 i$ s.t.\;$T_{i}=0,T^{'}_{i}=1$.  Note that $T^{''}\leq T$ or $T^{''}> T$
for any $T^{''}\in \psi(Q)$\;(note also that $T^{''}> T^{'}$ or $T^{''}\leq T^{'}$
for any $T^{''}\in \psi(Q)$). Since $T^{'}<T(i)\leq T$ we get $T=T(i)$.
 If $T(j)\leq T=T(i)$ then $l(i,j)\leq l(j,j)=0$, and if $T(j)>T(i)$
 then $l(j,i)\leq l(i,i)=0$. So for any $j\leq n-1$ there is a path
 from $i$ to $j$ or a path from $j$ to $i$. This implies that
 $Q=(Q^{'}(2)\overrightarrow{\amalg }Q^{'}(1))^{\circ}$, where $Q^{'}(1)\in \mathcal{A}$\;(resp.\;$Q^{'}(2)$) is the full sub-quiver of $Q$
 with $Q^{'}(1)_{0}=\{j\mid j\neq i, l(j,i)=0\}$\;(resp.\;$Q^{'}(2)_{0}=\{j\mid  l(i,j)=0\}$)\;(remark.\;we use the fact $Q\in \mathcal{S}$). 
 Let $Q(i):=Q^{'}(i)^{\circ}$, then   
 $(Q^{'}(2)\overrightarrow{\amalg }Q^{'}(1))^{\circ}\sim \phi(Q(2),Q(1))$.
 
 Now it is sufficient to show that $\psi(Q(i))=K(i)$. Consider a injection
 \[\iota:\psi(Q(1))\rightarrow K(1),\]
 where $\iota(T^{''})_{j}:=\left\{\begin{array}{cc}
 T^{''}_{j} & j\in Q(1)_{0} \\ 
 0 & \mathrm{otherwise}.
 \end{array}\right.$
 
 Let $T^{''}\in K(1)$, then $T^{''}_{j}\leq T^{''}_{j^{'}}+l_{Q}(j^{'},j)=T^{''}_{j^{'}}+l_{Q(1)}(j^{'},j)$ for any
 $j,j^{'}\in Q(1)_{0}$. Since $T^{''}\geq T(i)$ we get $T^{''}_{j}\leq T(i)_{j}=l_{Q}(i,j)=0$ for any $j\in Q(2)_{0}$ 
 and this implies $T^{''}\in \iota (\psi(Q(1))).$ So we get $\psi(Q(1))=K(1)$.

Similarly we can see that $\psi(Q(2))=K(2)$.

\end{proof}

Now consider the following properties for full sub-quiver $K$ of $C^{n}$,\\
(i)$_{n}$  $(0,\cdots 0),(1,\cdots,1)\in K_{0}$,\\
(ii)$_{n}$\ for any $T>T^{'}$ in $K$, there is a path from $T$ to $T^{'}$ in $K$,\\
(iii)$_{n}$ $<T,T^{'}>_{\pm}\in K_{0}$ for any $T,T^{'}\in K_{0}$, where
$<T,T^{'}>_{+}:=(\mathrm{min}\{T_{i},T^{'}_{i}\})_{i}$ and $<T,T^{'}>_{-}:=(\mathrm{max}\{T_{i},T^{'}_{i}\})_{i}$.\\
Let $\mathcal{L}_{n}:=\{K\in \mathcal{B}\mid K\ \mathrm{satisfies}\ \mathrm{(i)}_{n},\mathrm{(ii)}_{n},\mathrm{(iii)}_{n}\}$ 
and $\mathcal{L}:=\amalg \mathcal{L}_{n}.$  
For any $K\in \mathcal{L}$ we define $\psi^{-}(K)\in \mathcal{S}$ as follows. Let $Q^{'}(K)$ be a Hasse-quiver of the poset $(\{1,2,\cdots n-1\},\leq_{K})$, where $i\leq_{K} j\stackrel{\mathrm{def}}{\Leftrightarrow}T_{i}\geq T_{j}$ for any $T\in K$ (Remark. Assume $T_{i}=T_{j}$ for any $T\in K$. Now there is  a path $(0,\cdots 0)=T^{0}\rightarrow T^{1}\rightarrow \cdots \rightarrow T^{r}=(1,\cdots 1)$ in $K$ and so $\exists a$ s.t.\;$T^{a-1}_{i}=T^{a-1}_{j}=0$ and
 $T^{a}_{i}=T^{a}_{j}=1$. This implies $i=j$.). And set $\psi^{-}(K):=(\{\stackrel{s}{\circ}\}\overrightarrow{\amalg}Q^{'}(K))^{\circ}$.
 \begin{lem}
 \label{sl4'}
Let $Q\in \mathcal{A}^{\circ}$ then $\psi(Q)\in \mathcal{L}$.   
 \end{lem}
 \begin{proof}
 Let $K=\psi(Q)$ with $Q\in\mathcal{A}^{\circ}$ and $n:=\#Q_{0}$.
Then we have already seen that $K$ is a full sub-quiver of $C^{n-1}$ and 
$(0,\cdots, 0),(1,\cdots, 1)\in K_{0}$. So $K$ satisfies the condition 
(i)$_{n-1}$. Note that $K$ also satisfies the condition (ii)$_{n-1}$\;(see the proof of Theorem\;\ref{t}\;(3)).

Now it is sufficient to prove that $K$ satisfies the condition (iii)$_{n-1}$. Let $T,T^{'}\in K_{0}$ and $i,j\in Q_{0}$.
If min$\{T_{i},T^{'}_{i}\}=0$ or $l_{Q}(j,i)=1$, then min$\{T_{i},T^{'}_{i}\}\leq $min$\{T_{j},T^{'}_{j}\}+l_{Q}(j,i)$.
Assume that min$\{T_{i},T^{'}_{i}\}=1$ and $l_{Q}(j,i)=0$.
In this case $1=T_{i}\leq T_{j}$ and $1=T^{'}_{i}\leq T^{'}_{j}$ and this implies
 $T_{j}=T^{'}_{j}=1$. In particular min$\{T_{i},T^{'}_{i}\}\leq $min$\{T_{j},T^{'}_{j}\}+l_{Q}(j,i)$ for any $i,j\in Q_{0}$.   
So $<T,T^{'}>_{+}\in K_{0}$. Similarly we can show that $<T,T^{'}>_{-}\in K_{0}$.
 \end{proof}
 \begin{lem}
 \label{sl4}
 We get $Q\sim \psi^{-}(\psi(Q))$ for any $Q\in \mathcal{A}^{\circ}.$ In particular $\psi/\sim$ is injective.
 \end{lem}
 \begin{proof}
 It is sufficient to show that if $Q\in \mathcal{S}$ then $Q=\psi^{-}(\psi(Q))$. Let $s\neq i\rightarrow j$ in $Q$. Then 
 $T_{i}\leq T_{j}+l_{Q}(j,i)=T_{j}$ for any $T\in \psi(Q)$.
 Now assume that $j<_{\psi(Q)}j^{'}<_{\psi(Q)} i$ then 
 $l(j^{'},i)=T(j^{'})_{i}\leq T(j^{'})_{j^{'}}=0$ and
 $l(j,j^{'})=T(j)_{j^{'}}\leq T(j)_{j}=0$. This implies that
  there exists a path 
  \[i\rightarrow \cdots \rightarrow j^{'}\rightarrow\cdots \rightarrow j\]
  in $Q$. Since $Q\in \mathcal{S}$, this is a contradiction. So $i\rightarrow j$ in $Q^{'}(\psi(Q))$.
  
  On the other hands if $i \rightarrow j$ in $Q^{'}(\psi(Q))$ then
  $l(j,i)=T(j)_{i}\leq T(j)_{j}=0$. So there is a path 
  \[i\rightarrow j^{'}\cdots \rightarrow j\]
  in $Q$.
  Now $T_{i}\leq T_{j^{'}}+l(j^{'},i)=T_{j^{'}}\leq T_{j}+l(j,j^{'})=T_{j}$ for any $T\in \psi(Q).$ This implies 
  \[j\leq_{\psi(Q)}j^{'}<_{\psi(Q)}i,\]
  In particular $j=j^{'}$ and there is an arrow $i\rightarrow j$ in $Q$.
 
 So $Q\setminus \{s\}=Q^{'}(\psi(Q))$ and this implies $Q=(\{\stackrel{s}{\circ}\}\overrightarrow{\amalg}Q^{'}(\psi(Q)))^{\circ}=\psi^{-}(\psi(Q)).$  
 \end{proof}

Let $Q\in \mathcal{A}^{\circ}$ and $\Lambda(Q):=(Q_{0}\setminus \{s\},\leq_{Q})$ be a poset where, \[i\leq_{Q} j\stackrel{\mathrm{def}}{\Leftrightarrow} \mathrm{there\ is\ a\ path\ from }\ i\ \mathrm{to}\ j.\] We note that above proof shows 
$\leq_{Q}=\leq_{\psi(Q)}$. Let $\mathcal{I}(Q)$ be the set of the poset ideals of $\Lambda(Q)$.
\begin{prop}
\label{poset}
There is a poset isomorphism
\[(\mathcal{I}(Q),\subset )\simeq (\mathrm{lk}_{\mathrm{p}}(P(s_{Q}),\leq).\]

\end{prop}
\begin{proof}
Let $I\in \mathcal{I}(Q)$. We define a map $r_{I}:Q_{0}\setminus \{s_{Q}\}\rightarrow \{0,1\}$ as follows,
\[r_{I}(i):=\left\{\begin{array}{ll}
0 & \mathrm{if}\ i\in I \\ 
1 & \mathrm{if}\ i\notin I.
\end{array}\right.  \]
Now we show that $\rho: I\rightarrow P(s_{Q})\oplus_{i\in Q_{0}\setminus \{s_{Q}\}}\tau^{-r_{I}(i)}P(i)$ induces poset isomorphism \[(\mathcal{I}(Q),\subset )\simeq (\mathrm{lk}_{\mathrm{p}}(P(s_{Q}),\leq).\]
 
Let $I\in \mathcal{I}(Q)$. First we shows $r_{I}(i)\leq r_{I}(j)+l_{Q}(j,i)$. We only consider the case $r_{I}(i)=1,\;r_{I}(j)=0$. In this case $i\notin I$, $j\in I$ and this implies
there is no path from $i$ to $j$ in $Q$. So we get $l_{Q}(j,i)=1$.   
Now it is obvious that $\rho$ induces poset inclusion
\[(\mathcal{I}(Q),\subset )\rightarrow (\mathrm{lk}_{\mathrm{p}}(P(s_{Q}),\leq).\] So it is sufficient to show that $\rho$ is surjective. Let $T=P(s_{Q})\oplus_{i\in Q_{0}\setminus \{s_{Q}\}}\tau^{-r_{i}}P(i)\in \mathrm{lk}_{\mathrm{p}}(P(s_{Q})$ and $I(T):=\{i\in Q_{0}\setminus\{s_{Q}\}\mid r_{i}=0\}$. Then it is easy to check
that $I(T)\in \mathcal{I}(Q)$. This implies that $\rho(I(T))=T$.  
\end{proof}

\begin{cor}
Let $\Lambda$ be a finite poset, $\overrightarrow{\Lambda}$ be a its Hasse quiver, $\mathcal{I}(\Lambda)$ be the set
of  poset ideals of $\Lambda$ and $\overrightarrow{\mathcal{I}}(\Lambda)$ be a Hasse-quiver of the poset $(\mathcal{I}(\Lambda),\subset)$. Then  $\overrightarrow{\mathcal{I}}(\Lambda)=\psi(Q(\Lambda))$, where $Q(\Lambda):=(\{s\}\overrightarrow{\amalg}\overrightarrow{\Lambda})^{\circ}$.

\end{cor}

\begin{lem}
\label{sl5}
 Let $K\in \mathcal{B}$. Then the followings are equivalent,\\
 $(1)$ $K=\psi(Q)$ for some $Q\in \mathcal{A}^{\circ}$,\\
 $(2)$ $K\in \mathcal{L}$.
 \end{lem}
 \begin{proof}
 $((1)\Rightarrow (2))$:This is followed from Lemma\;\ref{sl4'}.

$((2)\Rightarrow (1))$:Let $K\in \mathcal{L}_{n}$. It is sufficient to show $K_{0}=\psi(\psi^{-}(K))_{0}$. 

First let $T\in K_{0}$ and $i,j\in \psi^{-}(K)_{0}$. If $l(j,i)=l_{Q(K)}(j,i)=0$ then $j\leq_{K} i$ and this implies $T_{i}\leq T_{j}$.
If $l(j,i)=1$ then  $T_{i}\leq T_{j}+l(j,i)$. So $T\in \psi(\psi^{-}(K))_{0}$.

Next assume that $\psi(\psi^{-}(K))_{0}\setminus K_{0}\neq \emptyset$ and let $T$ be a 
minimal element of $\{T\in \psi(\psi^{-}(K))_{0}\setminus K_{0}\mid T^{'}\rightarrow T\ \mathrm{for\ some}\ T^{'}\in K_{0}\}$. Let $T^{'}\in K_{0}$ with\; $T^{'}\rightarrow T$, then the conditions (i)$_{n}$ and (ii)$_{n}$ implies that
$\exists T^{''}\in K_{0}$ s.t.\;$T^{'}\rightarrow T^{''}$. Now $\exists i,j$
s.t.\;$0=T_{j}<T_{i}=1,\ T^{'}_{i}=T^{'}_{j}=0$ and $0=T^{''}_{i}<T^{''}_{j}=1$. By Lemma\;\ref{sl4} $\leq_{K}=\leq_{\psi(\psi^{-}(K))}$ and so $T_{i}>T_{j}$ implies
$\exists\;S\in K_{0}$ s.t.\;$S_{i}>S_{j}$. Let $T^{'''}:=<T,T^{''}>_{-}$, then
minimality of $T$ implies $T^{'''}\in K_{0}$. Now
\[(<<T^{'},S>_{-},T^{'''}>_{+})_{a}=\left\{\begin{array}{cc}
\mathrm{min}\{\mathrm{max}\{T_{a},S_{a}\},T_{a}\} & a\neq i,j \\ 
1 & a=i \\ 
0 & a=j.
\end{array}\right. \] 
So $T=<<T^{'},S>_{-},T^{'''}>_{+}\in K_{0}$. This is a contradiction.
  
 \end{proof}
 
 \begin{cor}
 \label{sc1}
 $\psi$ induces a bijection between $\mathcal{S}$ and $\mathcal{L}$.
 \end{cor}
 Now, by applying Theorem\;\ref{t}, we get the following result. 
 \begin{thm}
 \label{st1}
$(1)$ For any $Q\in \mathcal{A}^{\circ}$, there exists $K\in \mathcal{L}$ such that $\overrightarrow{\mathcal{T}}_{\mathrm{p}}(Q)=\overrightarrow{\amalg}K$.\\
$(2)$ For any $K\in \mathcal{L}$, there exists $Q\in \mathcal{A}^{\circ}$ such that $\overrightarrow{\mathcal{T}}_{\mathrm{p}}(Q)=\overrightarrow{\amalg}K$. 
 \end{thm} 
\begin{cor}
Let $Q(1),Q(2)\in \mathcal{A}^{\circ}$. Then the followings are equivalent,
$(1)$\;$\overrightarrow{\mathcal{T}}_{\mathrm{p}}(Q(1))=\overrightarrow{\mathcal{T}}_{\mathrm{p}}(Q(2))$,\\
$(2)$\;$\exists Q\in \mathcal{A}^{\circ}$\;s.t.\;$Q(1)\sim (Q\overrightarrow{\amalg}
Q\overrightarrow{\amalg}\cdots \overrightarrow{\amalg}Q)^{\circ}$ and $Q(2)\sim (Q\overrightarrow{\amalg}
Q\overrightarrow{\amalg}\cdots \overrightarrow{\amalg}Q)^{\circ}$.
\end{cor}  

\begin{proof}
$((2)\Rightarrow (1))$:This is followed by proposition\;\ref{sp1}, Lemma\;\ref{sl1} and Lemma\;\ref{sl2}.

$((1)\Rightarrow(2))$:Let 
\[\psi(Q(i))=S^{1}(i)\overrightarrow{\amalg}S^{2}(i)\overrightarrow{\amalg}\cdots \overrightarrow{\amalg} S^{r_{i}}(i)\ (S^{t}(i)\in \mathcal{B})\]
be a decomposition with $r_{i}$ being maximal\;$(i=1,2)$ and $r:=\mathrm{gcd}(r_{1},r_{2})$.

Consider a homomorphism $f:\mathbb{Z}\rightarrow \mathbb{Z}/r_{1}\mathbb{Z}\oplus \mathbb{Z}/r_{2}\mathbb{Z}$ where $f(t)=(t\mod r_{1},t\mod r_{2})$.
Let $1\leq a \leq r_{1}$ and $1\leq b \leq r_{2}$.
Then the condition\;$(1)$ implies 
\[a\equiv b\;\mathrm{mod}\; r \Rightarrow (a\;\mathrm{mod}\;r_{1},b\;\mathrm{mod}\; r_{2})\in \mathrm{Im\;}f \Rightarrow S^{a}(1)=S^{b}(2).\] So $S^{x+tr}(1)=S^{x}(2)=S^{x}(1)$ and $S^{x+tr}(2)=S^{x}(1)=S^{x}(2)(x\leq t).$ In particular we get,
\[\psi(Q(1))=S\overrightarrow{\amalg}S\cdots \overrightarrow{\amalg}S,\; \psi(Q(2))=S\overrightarrow{\amalg}S\cdots \overrightarrow{\amalg}S,\]
where $S=S^{1}(1)\overrightarrow{\amalg}S^{2}(1)\overrightarrow{\amalg}\cdots \overrightarrow{\amalg} S^{r}(1)$. By proposition\;\ref{sp1}, we can chose $Q$ satisfying
$\psi(Q)=S$. Now Lemma\;\ref{sl4} shows  $Q$ satisfies the condition\;$(2)$.  
\end{proof}

\section*{Acknowledgement}
The author would like to express his gratitude to Professor Susumu Ariki for his mathematical supports and warm encouragements.

\end{document}